\documentclass[11pt]{amsart}
\usepackage{graphicx}
  \usepackage{subfig}
        \usepackage{multicol}
        \usepackage{wrapfig}
        \usepackage{float}
\usepackage{amsthm, amsfonts, amsmath,color,hyperref}
\usepackage{comment}
\usepackage{verbatim}
\usepackage{paralist}
\usepackage{amsbsy,paralist}
\usepackage[backrefs]{amsrefs}
\usepackage[dvipsnames]{xcolor}
\usepackage{pdfpages}
\usepackage{url}
\usepackage{graphicx}
\usepackage{setspace}
\usepackage{enumitem}
\usepackage{mathtools}
\usepackage{fullpage}
\allowdisplaybreaks

\theoremstyle{definition}
\newtheorem{definition}{Definition}[section]

\newtheorem{theorem}{Theorem}[section]
\newtheorem{proposition}{Proposition}[section] 
\newtheorem{lemma}{Lemma}[section] 
\newtheorem{corollary}{Corollary}[section]

\newcommand{\N}{{\mathbb N}}

\newcommand{\alp}{{\sqrt{(1+m+n)^2-4mn}}}

\newcommand{\alpy}{{\sqrt{(1+(m+n)y)^2-4mny^2}}}

\title{Bin Decompositions}

\author{Daniel Gotshall}
\address{Department of Mathematics, Saint Peter's University}
\email{\textcolor{blue}{\href{mailto:dgotshall16@saintpeters.edu}{dgotshall16@saintpeters.edu}}}

\author{Pamela E. Harris}
\address{Department of Mathematics and Statistics, Williams College, United States}
\email{\textcolor{blue}{\href{mailto:peh2@williams.edu}{peh2@williams.edu}}}
\thanks{P.\,E. Harris was supported by NSF award DMS-1620202.}

\author{Dawn Nelson}
\address{Department of Mathematics, Saint Peter's University}
\email{\textcolor{blue}{\href{mailto:dnelson1@saintpeters.edu}{dnelson1@saintpeters.edu}}}

\author{Maria D. Vega}
\address{Department of Mathematical Sciences, United States Military Academy}
\email{\textcolor{blue}{\href{mailto:maria.vega@usma.edu}{maria.vega@usma.edu}}}

\author{Cameron Voigt}
\address{Department of Mathematical Sciences, United States Military Academy}
\email{\textcolor{blue}{\href{mailto:cdv1218@gmail.com}{cdv1218@gmail.com}}}

\begin{document}
\subjclass[2010]{11B39, 65Q30, 60B10}

\keywords{Zeckendorf decompositions, bin decompositions,  Gaussian behavior, integer decompositions}

\maketitle

\begin{abstract}
It is well known that every positive integer can be expressed as a sum of nonconsecutive Fibonacci numbers provided the Fibonacci numbers satisfy $F_n =F_{n-1}+F_{n-2}$ for $n\geq 3$, $F_1 =1$ and $F_2 =2$. 
In this paper, for any $n,m\in\mathbb{N}$ we create a sequence called the $(n,m)$-bin sequence with which we can define a notion of a legal decomposition for every positive integer. These sequences are  not always positive linear recurrences, which have been studied in the literature, yet we prove, that like positive linear recurrences, these decompositions exist and are unique. Moreover, our main result proves that the distribution of the number of summands used in the $(n,m)$-bin legal decompositions displays Gaussian behavior. 
\end{abstract}

\section{Introduction} 
In 1972 Edouard Zeckendorf proved that any positive integer can be uniquely decomposed as a sum of non-consecutive Fibonacci numbers provided we use the recurrence $F_1=1$, $F_2=2$, and $F_n=F_{n-1}+F_{n-2}$ for $n\geq 3$ \cite{Ze}.
Since then numerous researchers have generalized Zeckendorf's theorem to other recurrence relations \cites{miller,CFHMN1,DDKMMV,DDKMV, KKMW, lengyel}. Most work involved recurrence relations with positive leading terms, called positive linear recurrences. That was until Catral, Ford, Harris, Miller, and Nelson generalized these results to the $(s,b)$-Generacci sequences and to the Fibonacci Quilt sequence, which are defined by non-positive linear recurrences \cites{CFHMN1,CFHMN2,newbehavior}, and Dorward, Ford, Fourakis, Harris, Miller, Palsson, and Paugh to the  $m$-gonal sequences, which arise from a geometric construction via inscribed $m$-gons \cites{mgonpaper,individualgaps}. 
The main results in these studies involved determining the uniqueness of the decompositions of nonnegative integers using the numbers in these new sequences, determining whether the behaviour arising from the average number of summands in these decompositions is Gaussian, and other related results.

A way to interpret the creation of the $(s,b)$-Generacci sequences is to imagine an  infinite number of bins each containing $b$ distinct positive integers. Given a number $\ell\in\mathbb{N}$, we decompose it as a sum of elements in the sequence such that their sum is $\ell$, and the terms satisfy that 1) no two numbers in the sequences used in the decomposition appear in the same bin, and that 2) we do not use numbers in $s$ bins to the left and right of any bin containing a summand used in the decomposition of $\ell$. If such a decomposition of $\ell$ exists using the numbers in the sequence, we then say that $\ell$ has a legal decomposition. If every positive integer $\ell$ has a legal decomposition, then we call the sequence of numbers satisfying this property the $(s,b)$-Generacci sequence. Note that the $(1,1)$-Generacci sequence gives rise to the Fibonacci sequence, as we have bins with only one integer and we cannot use any consecutive integers in any decomposition. 

Motivated by the bin construction used in the $(s,b)$-Generraci sequences, we create the $(n,m)$-bin sequences. These sequences are defined by nonpositive linear recurrences and depend on the positive integer parameters $s,b$ for Generacci sequences and $n,m$ for bin sequences.
The terms of an  $(n,m)$-bin sequence $\{a_x\}_{x=0}^\infty$ can be pictured via
\begin{equation}
\underbracket{ \underbracket{a_0,\ldots,a_{n-1}}_{n}\ ,\underbracket{a_{n},\ldots,a_{n+m-1}}_{m}}_{\mathcal{B}_0}
\ ,\ldots,\ 
\underbracket{\underbracket{a_{(n+m)k},\ldots,a_{(n+m)k+n-1}}_{n}\ ,\underbracket{a_{(n+m)k+n},\ldots,a_{(n+m)k+n+m-1}}_{m}}_{\mathcal{B}_{k}}\ ,\ldots.\end{equation}
Note that the first term  in the sequence is indexed by 0. Notice also that there are $n$ terms in the first bin and $m$ terms in the next. The number of terms in each subsequent bin alternates between $n$ and $m$. We use the notation $\mathcal{B}_k$ to indicate a pair of bins of size $n$ and $m$, in that order. 
Given a term in the sequence, $a_x$, we can determine which $\mathcal{B}_k$ contains $a_x$ and whether $a_x$ is in the $n$ or $m$ sized bin by using the division algorithm to write $x=(n+m)k+i$. If $0\leq i\leq n-1$ then $a_x$ is in the $n$ sized bin. If $n\leq i\leq m+n-1$ then $a_x$ is in the $m$ sized bin. For example,
consider the (2,3)-bin sequence and term $a_{44}$. Since $44 = (2+3)8+4$, $a_{44}\in \mathcal{B}_8$ and since $i=4\geq 2=n$, $a_{44}$ is the third term in the $m=3$ size bin.

Before defining how we construct the sequences, we need to establish the notion of a legal decomposition.

\begin{definition}
Let an increasing sequence of integers $\{a_i\}_{i=0}^\infty$, divided into bins of sizes $n$ and $m$ be given. For any $n,m\in\mathbb{N}$, a {\em $(n,m)$-bin legal decomposition} of an integer using summands from this sequence is a decomposition in which no two summands are from the same or adjacent bins. 
\end{definition}

As described in \cite{DDKMMV}, this notion of legal decompositions is an $f$-decomposition defined by the function $f:\N_0\rightarrow\N_0$ with 
\begin{align}f(j)=\begin{cases}
m+i & \mbox{if } j\equiv i\!\!\!\mod m+n \quad\mbox{ and
} 0\leq i\leq n-1\\
i & \mbox{if } j\equiv i\!\!\!\mod m+n  \quad\mbox{ and } n\leq i\leq m+n-1.
\end{cases}\label{eq:fdef}
\end{align}
In other words, if $a_j$ is a summand in a $(n,m)$-bin legal decomposition, then none of the previous $f(j)$ terms ($a_{j-f(j)},a_{j-f(j)+1}, \ldots, a_{j-1}$) are in the decomposition. 
Consider the (2,3)-bin legal decompositions. Then $f:\N_0\rightarrow\N_0$ is the periodic function
\[ \{f(j)\}=\{ 3, 4, 2, 3, 4, 3, 4, 2, 3, 4,\ldots     \}.  \]
Note $f(44)=4$, so if $a_{44}$ is a term in an $(n,m)$-bin legal decomposition, then  $a_{40},a_{41},a_{42}, a_{43}$ are not in the decomposition. Notice that $a_{42}, a_{43}$ are other terms in the 3-bin (denoting the bin of size 3) that contains $a_{44}$ and that $a_{40},a_{41}$ are the two terms in the previous 2-bin (denoting the bin of size 2).

Through an immediate application of Theorems 1.2 and 1.3 from \cite{DDKMMV} we can establish that for any $n,m\in\mathbb{N}$, $(n,m)$-bin legal decompositions are unique and we get Proposition \ref{1.1}.
\begin{proposition}\label{1.1}
For each pair of $n,m\in\mathbb{N}$ there is a unique sequence  such that every positive integer has a unique $(n,m)$-bin legal decomposition.
\end{proposition}

With this result at hand, we can now formally define an ($n,m$)-bin sequence.
\begin{definition}
For each pair of $n,m\in\mathbb{N}$, an {\em ($n,m$)-bin sequence} is the unique sequence such that every positive integer has a unique $(n,m)$-bin legal decomposition.
\end{definition}
Using this definition one can verify that the $(2,3)$-bin sequence begins:
\[\underbracket{1, 2},\underbracket{ 3, 4, 5}, \underbracket{6, 9}, \underbracket{12, 18, 24}, \underbracket{30, 42}, \underbracket{54, 84, 114}, \underbracket{144, 198}, \underbracket{252, 
396, 540},\underbracket{684, 936}, \underbracket{1188, 
1872, 2556}, \ldots \]
and that the $(2,3)$-bin legal decomposition of 2018 is $2018=1872+144+2$. We also note that we can once again recover the Fibonacci sequence, which in this case is given by the $(1,1)$-bin sequence.

In Section \ref{sec:recurrences} we establish a recurrence for the $(n,m)$-bin sequences.
\begin{theorem}\label{thm:single recurrence} Assume $\{a_x\}_{x=0}^\infty$ is an $(n,m)$-bin sequence. Then for all $n,m\geq 1$ and  $x\geq 2(m+n)$, 
\begin{align}
a_x&=(m+n+1)a_{x-(m+n)}-mna_{x-2(m+n)}.\label{eq:single recurrence}
\end{align}
\end{theorem}
We note that the recurrence above is sometimes a PLR and sometimes it is not. For example, as noted previously, the $(1,1)$-bin legal decompositions are exactly the  Zeckendorf decompositions, and use the Fibonacci numbers, which are defined via a PLR. However, when $n=2$ and $m=1$ the recurrence above is not a PLR and we show this in Appendix \ref{appendix}. This provides further motivation to study sequences that are more broadly defined and do not necessarily fall under (or out of) the PLR definition.

Our main result establishes that the  number of summands used in $(n,m)$-bin legal decompositions of the natural numbers follows a  Gaussian distribution.
\begin{theorem}[Gaussian Behavior of Summands]\label{thm:gaussian}
Let the random variable $Y_k$ denote the number of summands in the (unique) $(n,m)$-bin legal decomposition of an integer chosen uniformly at random from $[0, a_{(n+m)k})$. Normalize $Y_k$ to $Y_k' = (Y_k - \mu_k)/\sigma_k$, where $\mu_k$ and $\sigma_k$ are the mean and variance of $Y_k$ respectively. Then  
\begin{align}\label{muConstantA} \mu_k  \ = \   Ck+O(1), \ \ \ \ \sigma_k^2  \ = \   C'k+O(1),
 \end{align} for some positive constants 
 \[ C=\frac{\alp-1}{\alp},\qquad C'=\frac{(m+n)(1+m+n)-4mn}{\alp^3}.   \]
  Moreover, $Y_k'$ converges in distribution to the standard normal distribution as $k \rightarrow \infty$.
\end{theorem}

As we noted earlier, the $(1,1)$-bin sequence is simply the Fibonacci sequence. In this case, the formulas for the mean and the variance given in \eqref{muConstantA} simplify to the known formulas obtained by Lekkerkerker \cite{Lek} and Kol\v{o}glu et al.~\cite{KKMW}. Lekkerkerker computed that for $x\in[F_n,F_{n+1})$ the average number of summands in a Zeckendorf decomposition is $\frac{n}{\phi^2+1}+O(1)$, where $\phi = \frac{1+\sqrt{5}}{2}$. The result is the same when the interval is extended to  $x\in[0,F_n)$. 
In \cite{KKMW}, the authors show that for $x\in[F_n,F_{n+1})$ the variance of the number of summands in a Zeckendorf decomposition is $\frac{\phi n}{5(\phi+2)}+O(1)$. Again the result is same when the interval is extended to  $x\in[0,F_n)$.

\begin{corollary}\label{cor:m=n=1}
Consider the $(1,1)$-bin sequence. For $x\in [0,a_{2k})$ the average and variance of the number of summands in a $(1,1)$-bin legal decomposition is 
\[ \mu_k = \frac{\sqrt{5}-1}{\sqrt{5}}k+O(1)= \frac{1}{\phi^2+1}2k+O(1)  \hspace{.25in}\hbox{and}\hspace{.25in}
 \sigma^2_k = \frac{2}{5\sqrt{5}}k+O(1)=   \frac{\phi }{5(\phi+2)}2k+O(1). \]
\end{corollary}

The paper is organized as follows. Section \ref{sec:recurrences} establishes needed recurrence relations and proves Theorem \ref{thm:single recurrence}, Section \ref{sec:generating} develops helpful generating functions, and Section \ref{sec:gaussian} pulls these ideas together and contains the proof of Theorem \ref{thm:gaussian}. We end with some directions for future research.

\section{Recurrence relations}\label{sec:recurrences}

In this section we establish recurrence relations for $(n,m)$-bin sequences.
We will establish Theorem \ref{thm:single recurrence} via the following two technical results. Lemma \ref{rec_reln_a} provides a family of recurrence relations. For example, equation \eqref{a0_rec} computes the first term in the $n$-bin, equation \eqref{ai_rec} computes the remaining terms in the $n$-bin and the first term in the $m$-bin, and equation \eqref{aj_rec} computes the remaining terms in the $m$-bin. In contrast, Theorem \ref{thm:single recurrence} provides a single recurrence relation that can be used to compute any term regardless of its position in the bins.

\begin{lemma}\label{rec_reln_a}
If $n,m\in\mathbb{N}$, then for $k\geq1$
\begin{align}
        a_{(m+n)(k+1)} & = a_{(m+n)k+m+n-1} +a_{(m+n)k}\label{a0_rec}\\
 a_{(m+n)(k+1)+i} & = a_{(m+n)(k+1)+(i-1)} +a_{(m+n)k+n} \qquad \mbox{for } 1\leq i\leq n\label{ai_rec}\\
    a_{(m+n)(k+1)+j} & = a_{(m+n)(k+1)+j-1} +a_{(m+n)(k+1)}\qquad \mbox{for } n+1\leq j\leq m+n-1\label{aj_rec}
\end{align}
\end{lemma}

\begin{proof} By Theorems 1.2 and 1.3 in \cite{DDKMMV}, $a_x=a_{x-1}+a_{x-1-f(x-1)}$. When $x=(m+n)(k+1)$, then $x-1 = (m+n)k+m+n-1$ and $f((m+n)k+m+n-1)=m+n-1$. Hence Equation (\ref{a0_rec}), is immediate.
The other equations follow from a similar argument.
\end{proof}

Lemma \ref{lem_same} interweaves the family of recurrence relations to show that if the single recurrence relation (of Theorem \ref{thm:single recurrence}) is true for $x\equiv 0\,(\mbox{mod }m+n)$, then  it is true for all $x$.

\begin{lemma}\label{lem_same}
Assume $n,m\geq 1$. If 
\begin{equation}\label{lem_rec}
a_x=(m+n+1)a_{x-(m+n)}-mna_{x-2(m+n)}\end{equation}
for $x\geq 2(m+n)$ and $x\equiv 0\,(\mbox{mod }m+n)$, then  Equation (\ref{lem_rec}) is true for all $x\geq 2(m+n)$. 
\end{lemma}

\begin{proof}
By hypothesis, \[a_{(m+n)k}=(m+n+1)a_{(m+n)k-(m+n)}-mna_{(m+n)k-2(m+n)}.\] In other words, \[a_{(m+n)k}=(m+n+1)a_{(m+n)(k-1)}-mna_{(m+n)(k-2)}.\]
So applying Equation (\ref{a0_rec}), we have 
\begin{align*}
 a_{(m+n)(k-1)+m+n-1} +a_{(m+n)(k-1)}=&(m+n+1)[a_{(m+n)(k-2)+m+n-1} +a_{(m+n)(k-2)}]\\
 &-mn[a_{(m+n)(k-3)+m+n-1} +a_{(m+n)(k-3)}]. 
\end{align*}
Thus 
\begin{align*}
 a_{(m+n)(k-1)+m+n-1} -[(m+n+1)a_{(m+n)(k-2)+m+n-1}-mna_{(m+n)(k-3)+m+n-1}]\\
 = - a_{(m+n)(k-1)}+[(m+n+1)a_{(m+n)(k-2)}-mna_{(m+n)(k-3)}]. 
\end{align*}
By hypothesis, the right hand side of this equation is 0. Hence so is the left side and thus Equation (\ref{lem_rec}) is true for $x\equiv m+n-1\,(\mbox{mod }m+n)$. 

Repeating a similar argument several more times shows that Equation (\ref{lem_rec}) is true for all $x$.
\end{proof}

It remains to prove that Equation \eqref{lem_rec} is true for $x\equiv 0\,(\mbox{mod }m+n)$. We do this in the following proof and thus establish Theorem \ref{thm:single recurrence}.

\begin{proof}[Proof of Theorem \ref{thm:single recurrence}]

Assume $\{a_x\}_{x=0}^\infty$ is an $(n,m)$-bin sequence. As explained in Section 1, this sequence is an $f$-sequence defined by the function $f(j)$ given in Equation \eqref{eq:fdef}. Note that the period of $f(j)$ is $m+n$ and $m+n\geq f(j)+1$ for all $j$.

By Theorem 1.5 in \cite{DDKMMV}, since $f(j)$ is periodic, we know that there is a single recurrence relation for our sequence. The proof of Theorem 1.5 in \cite{DDKMMV} gives us an algorithm for computing the single recurrence relation.

Consider $m+n$ subsequences of $\{a_x\}_{x=0}^\infty$ given by terms  whose indices are all in the same residue class mod $m+n$. We will begin by finding a recurrence relation for each subsequence:
\begin{equation}\label{coeff}
 a_x=\sum_{i=1}^{m+n+1}c_i a_{x-(m+n)i}.   \end{equation} 
A priori, these relations may be different for each residue class, but  Lemma \ref{lem_same} tells us that all relations are  in fact the same. Thus we focus on the subsequence corresponding to the 0 residue class. 

It remains to solve for the constants $c_i$ in \eqref{coeff}. To solve for these constants we will use linear algebra techniques, in particular we use matrices and vectors to represent systems of equations. 
Each of the equations in Lemma \ref{rec_reln_a} can be rewritten as vectors. (The starred columns, beginning with 0, are those that are indexed by multiples of $m+n$, and the columns marked with $\circ$ are indices congruent to $m$ modulo $m+n$): 
\[\begin{array}{rrrrrrrrrrrrrrrr}
&&\star\,&&&&\circ&&&&\star&&&&\circ\,\\
    \vec{v}_0 &=& [1,&-1,    &0, &  &  & & \ldots,&0,&-1,&0,&\ldots]\\
    \vec{v}_1 &=& [0,& 1,    &-1,&0,&  & &\ldots,&0,&-1,&0,&\ldots]\\
    \vdots \\
    \vec{v}_{m-1} &=& [0,&\ldots,& 0,& 1,&-1,&0,&\ldots,&0,&-1,&0,&\ldots]\\
    \vec{v}_{m} &=& [0,&\ldots,&&0,& 1,&-1,&0,&\ldots,&0,&\ldots,&&0,&-1]\\
    \vdots \\
\vec{v}_{m+n-1} &=& [0,&\ldots,& &&&&0,& 1,&-1,&0,&\ldots,&0,&-1]
\end{array}\]
Vector  $\vec{v}_0$ corresponds to the recurrence relation in (\ref{a0_rec}), $\vec{v}_1$ to $\vec{v}_{m-1}$  correspond to the recurrence relations in (\ref{aj_rec}), $\vec{v}_m$ to  $\vec{v}_{m+n-1}$ correspond to the recurrence relations in (\ref{ai_rec}). For all $\vec{v}_j$ the number of leading 0's is $j$ and the number of middle 0's is $f
(m+n-j)-1$. 

Define $T$ to be the transformation that  shifts all coordinates to the right by $(m+n)$ places.

According to the algorithm in \cite{DDKMMV} the goal is to zero out the coordinates that are not indexed by multiples of $m+n$ (the period). Note the first column is indexed by 0. Our first step in this process is to define $\vec{w}_1$, a linear combination of the $\vec{v}_j$. We have:
\[ \vec{w}_1 =\vec{v}_0+\cdots+\vec{v}_{m+n-1}=[1,\; 0,\; \ldots,\; 0,\; -m-1,\;0 ,\; \ldots,\; 0,\;-n, \;0], \]
where there are $(m+n-1)$ 0's in the first set and $(m-1)$ 0's in the second set. We continue and use $T$ to define $\vec{w}_2$:
\begin{align*}
 \vec{w}_2 &=\vec{w}_1+n\sum_{j=m}^{m+n-1}T\vec{v}_j\\
 &=[1,\; 0,\; \ldots,\; 0,\; -m-1,\;0 ,\; \ldots,\; 0, \;-n,\;0 ,\; \ldots,\; 0,\;-n^2 ], \end{align*}
where there are $(m+n-1)$ 0's in the first and second sets and $(m-1)$ 0's in the last set.

Note that in $\vec{w}_0=\vec{v}_0$, $\vec{w}_1$, and $\vec{w}_2$ the bad coordinates (the coordinates that are not 0 and not indexed by multiples of $(m+n)$) are given by 
\[\begin{array}{rrlrrrrr}
    \vec{u}_0 &= &[-1,&0 & \ldots, &0]\\
    \vec{u}_1 &= &[0, & \ldots,&0, &-n]\\
    \vec{u}_2 &= &[0, & \ldots,&0, &-n^2]
\end{array}.\]
We simplify by removing the common strings of 0's:
\[\begin{array}{rrlrrrr}
    \vec{u}_0 &= &[-1, &0]\\
    \vec{u}_1 &= &[0, &-n]\\
    \vec{u}_2 &= &[0, &-n^2]
\end{array}.\]
There exists a non-trivial solution to $\sum_{j=0}^2\lambda_j\vec{u}_j=0$, namely $\lambda_0 =0,\lambda_1 =-n, \lambda_2 =1$. Using these values, we can write a linear combination of the $\vec{w}_j$ in which we have succeeded in zeroing out the coordinates that are not multiples of $m+n$:
\[ \sum_{j=0}^2  \lambda_jT^{2-j}\vec{w}_j = [1,\, 0,\, \ldots,\, 0,\, -(m+n+1),\,0 ,\, \ldots,\, 0, \,mn,\,0 ,\, \ldots ].  \]
Thus Relation (\ref{coeff}) becomes $a_x = (m+n+1)a_{x-(m+n)}-mna_{x-2(m+n)}$. Note that a priori this is only the recurrence relation for the subsequence given by the terms whose indices are congruent to $0\, (\mbox{mod } m+n)$. Fortunately, applying Lemma \ref{lem_same}, we see that this recurrence relation is the single relation for the entire sequence.
\end{proof}

\section{Counting summands with generating functions}\label{sec:generating}

In this section we  provide generating functions for counting integers with a fixed number of  summands in their $(n,m)$-bin legal decomposition. We continue to assume throughout that  $\{a_x\}_{x=0}^\infty$ is an $(n,m)$-bin sequence.

Let $p_{k,c}$ denote the number of integer $z\in [0,a_{(n+m)k})$ whose legal decomposition contains exactly $c$ summands, where $c\geq 0$. Then by definition

\begin{equation}
p_{0,c}=\begin{cases}
   1   &c=0 \\
    0  & c>0
\end{cases}
\end{equation}
\begin{equation}
p_{1,c}=\begin{cases}
   1   &c=0 \\
    n+m  & c=1\\
    0 & c>1
\end{cases}
\end{equation}
for all $k\geq 0$, $p_{k,0}=1$. Also, for all $k\geq 0$,  $p_{k,1}=k(n+m)$. Moreover, for all $c>k\geq 0$, $p_{k,c}=0$. 

We also have the following recurrence relation for the values of $p_{k,c}$.
\begin{proposition}\label{prop:rec for pkc}
If $k\geq 2$ and $c\geq 0$, then 
\begin{equation}
p_{k,c}= p_{k-1,c} +(m+n)p_{k-1,c-1} -nmp_{k-2,c-2}.
\end{equation}
\end{proposition}

\begin{proof}
The decomposition of an integer $z\in [0,a_{(n+m)k})$ either has a summand from the bin $\mathcal{B}_{k-1}$ or it doesn't. If it doesn't then the number of integers with $c$ summands is $p_{k-1,c}$. 

If $z$ has a summand in the bin $\mathcal{B}_{k-1}$, then there are two possibilities: either the summand lies in the bin  of size $m$ or in the bin of size $n$. In what follows we need to recall that the first sub-bin of $\mathcal{B}_{k-1}$ has size $n$ and the second has size $m$.  If the largest summand appearing in the decomposition of $z$ is in the sub-bin of size $m$ then there are $m$ ways to choose it, and since the next largest legal summand is less than $a_{(n+m)(k-1)}$, there are $p_{k-1,c-1}$ ways to choose the remaining $c-1$ summands. Hence there are $mp_{k-1,c-1}$ integers with $c$ summands and largest summand from the $m$ sub-bin of $\mathcal{B}_{k-1}$. On the other hand, if the largest summand in the decomposition of $z$ is in the sub-bin of size $n$, the quantity $np_{k-1,c-1}$ overcounts by $nmp_{k-2,c-2}$, because a decomposition with a summand from the sub-bin of size $n$ of $\mathcal{B}_{k-1}$ and a summand from the sub-bin of size $m$ of $\mathcal{B}_{k-2}$ does not give rise to a $(n,m)$-bin legal decomposition.
Hence $p_{k,c}= p_{k-1,c} +(m+n)p_{k-1,c-1}-nmp_{k-2,c-2}$.
\end{proof}

\begin{proposition}\label{prop:Fxy}
Let $F(x,y)=\sum_{k\geq 0}\sum_{c\geq 0}p_{k,c}x^ky^c$ be the generating function of the $p_{k,c}$'s arising from the $(n,m)$-bin legal decompositions. Then
\begin{align}
    F(x,y)&=\frac{1}{1-x-(m+n)xy+mnx^2y^2}.\label{eq:Fxy}
\end{align}
\end{proposition}
\begin{proof}
Noting that $p_{k,c}=0$ if either $k < 0$ or $c < 0$, using explicit values of $p_{k,c}$ and the
recurrence relation from Proposition \ref{prop:rec for pkc}, after some straightforward algebra we obtain
\[F(x,y)=xF(x,y)+(m+n)xyF(x,y)-mnx^2y^2F(x,y)+1\]
from which \eqref{eq:Fxy} follows.
\end{proof}

\section{Gaussian behavior}\label{sec:gaussian}
To motivate the main result of this section, we point the reader to the following experimental observations. Taking samples of 100,000 integers from the intervals $[0, a_{10000(m+n)})$, in Figure \ref{gaussiangraphs} we provide a histogram for the distribution of the number of summands in the $(n,m)$-bin decomposition of these integers, when $(n,m)=(1,2)$, $(n,m)=(2,1)$, $(n,m)=(2,3)$, and $(n,m)=(3,2)$ respectively. In these figures we also provide the Gaussian curve computed using each sample's mean and variance. Furthermore, Table \ref{table:gaussian} gives the values of the predicted means and variances as computed using Theorem \ref{thm:gaussian}, as well as the sample means and variances, for each of the samples considered.

\begin{figure}[h]
  \centering
  \subfloat[$(n,m)=(1,2)$]{\label{gaussian(n,m)=(1,2)}\includegraphics[width=6.0 cm]{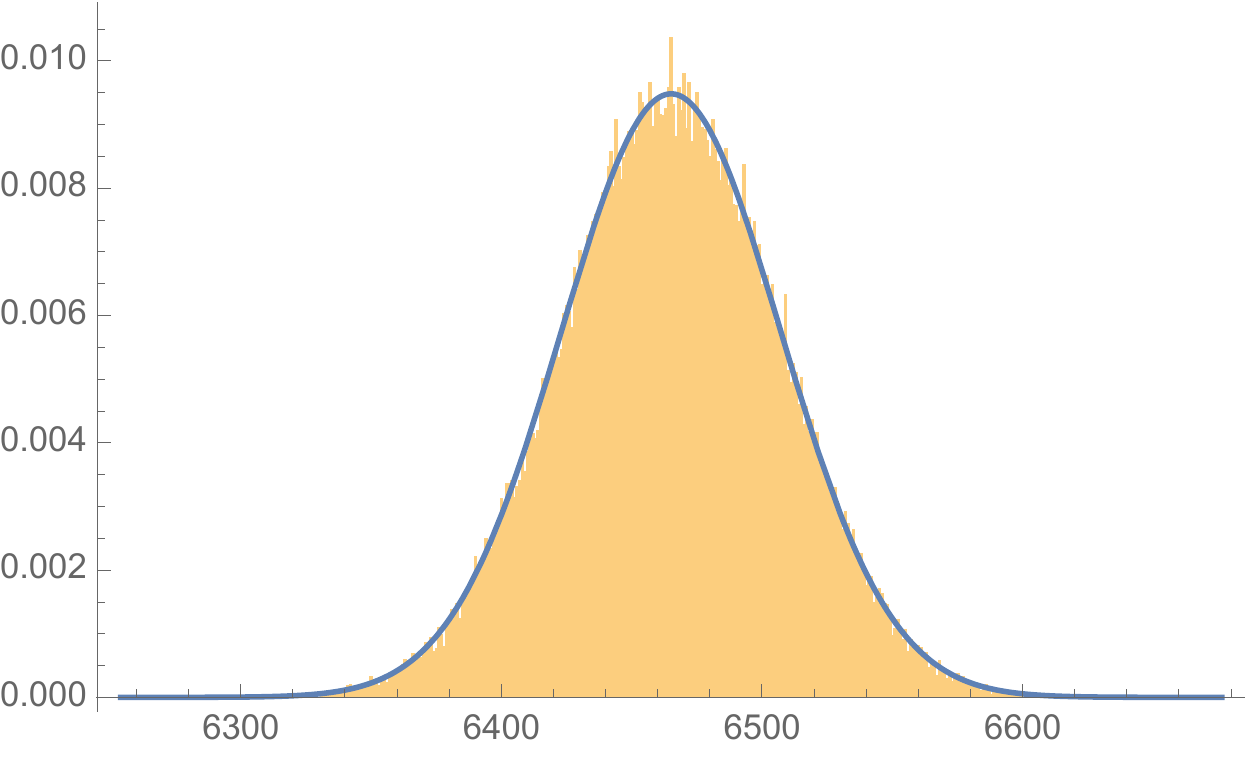}}\hspace{15 mm}
  \subfloat[$(n,m)=(2,1)$]{\label{gaussian(n,m)=(2,1)}\includegraphics[width=6.0 cm]{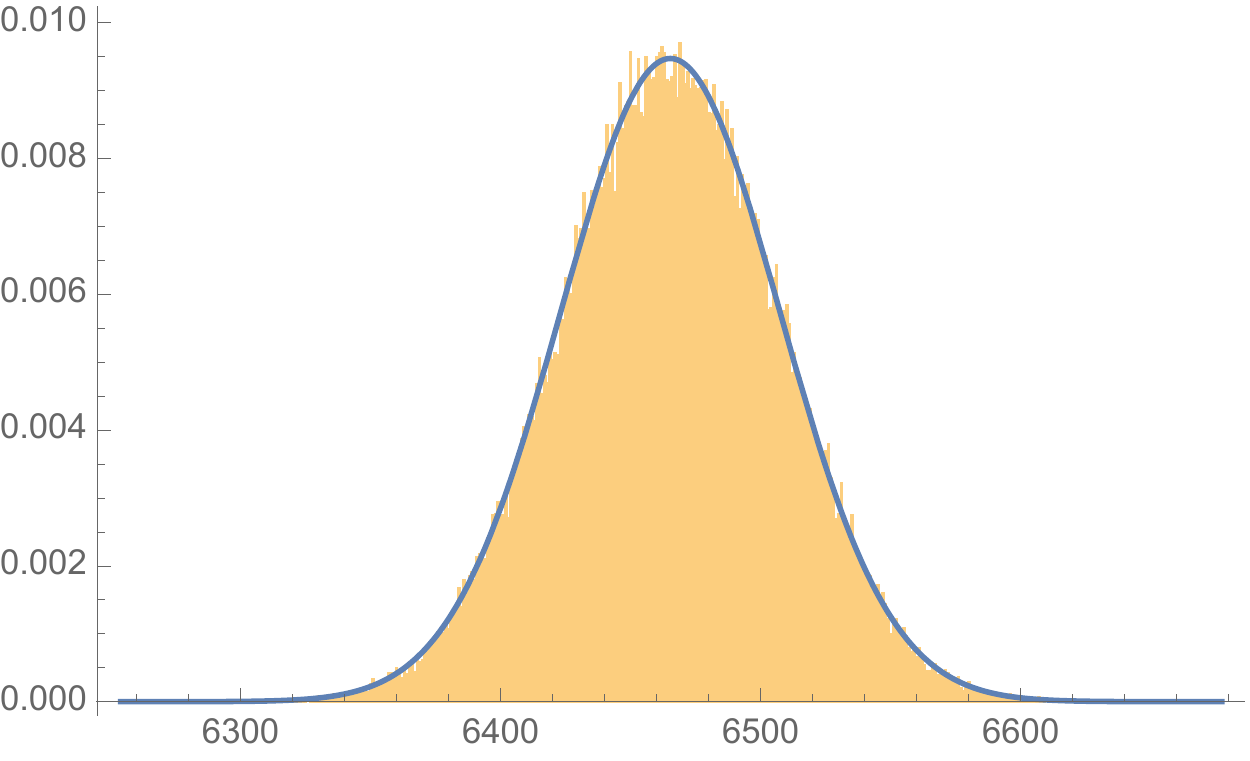}}\\
  \subfloat[$(n,m)=(2,3)$]{\label{gaussian(n,m)=(2,3)}\includegraphics[width=6.0 cm]{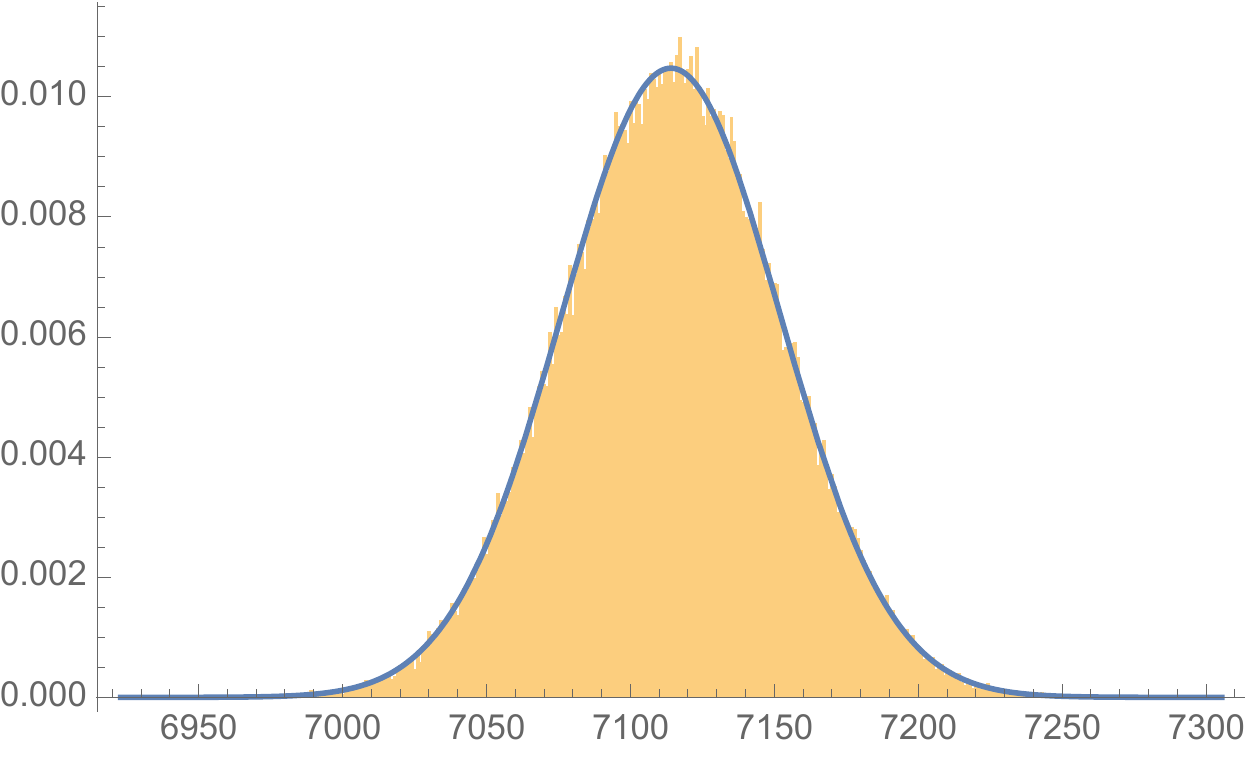}}\hspace{15 mm}
  \subfloat[$(n,m)=(3,2)$]{\label{gaussian(n,m)=(3,2)}\includegraphics[width=6.0 cm]{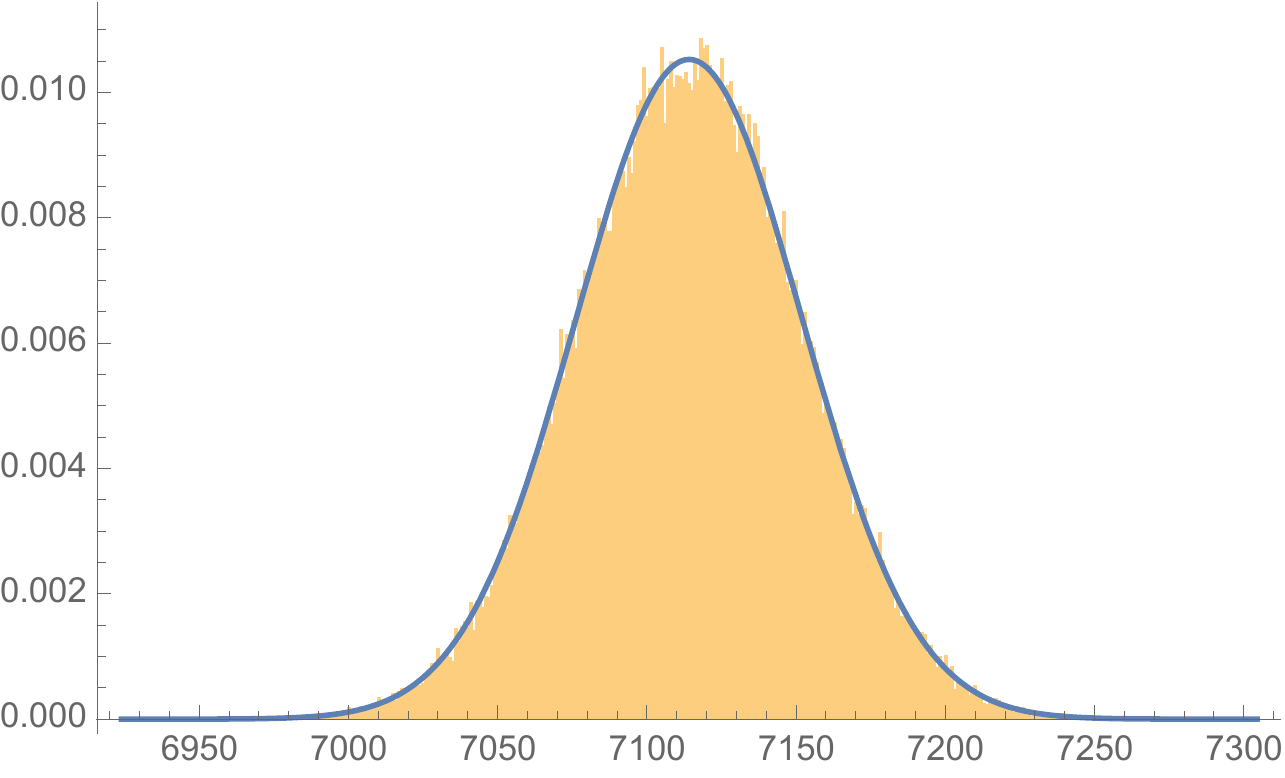}}
  \caption{Distributions for the number of summands in the $(n,m)$-bin decomposition for a random sample of 100,000 integers from the intervals $[0, a_{10000(m+n)})$.}
  \label{gaussiangraphs}
\end{figure}

\begin{table}[h]
\centering
\begin{tabular}{|c|c|c|c|c|c|}
\hline
Figure&$(n,m)$&Predicted Mean&Sample Mean&Predicted Variance&Sample Variance\\\hline\hline
\ref{gaussian(n,m)=(1,2)}&$(1,2)$&$6464.466094$& $6465.205230$& $1767.766953$& $1770.751318$\\\hline
\ref{gaussian(n,m)=(2,1)}&$(2,1)$&$6464.466094$&$6465.418910$&\ $1767.766953$&\  $1774.385128$\\\hline
\ref{gaussian(n,m)=(2,3)}&$(2,3)$&$7113.248654$&$7114.140920$&\ $1443.375673$&\ $1450.656668$\\\hline
\ref{gaussian(n,m)=(3,2)}&$(3,2)$&$7113.248654$&$7114.202700$&\ $1443.375673$&\ $1437.312966$\\\hline
\end{tabular}
\caption{Predicted means and variances versus sample means and variances for simulations from Figure  \ref{gaussiangraphs}.}
\label{table:gaussian}
\end{table}

From these observations one might speculate that for any pair of integers $n,m\in\mathbb{N}$ the distribution of the number of summands in the $(n,m)$-bin legal decompositions of integers in the interval $[0,a_{(n+m)k})$ displays Gaussian behavior. This is in fact the statement of Theorem \ref{thm:gaussian}.

To prove Theorem \ref{thm:gaussian} we first need the following technical results.

\begin{lemma}
For all $m,n,y>0$, the following inequalities hold:
\begin{align}
    (1+(m+n)y)^2-4mny^2&>1+(m+n)y\label{lem:betay}\\
         \alpy &>1\label{lem:discr}\\
    1+(m+n)y+\alpy &>1+(m+n)y-\alpy >  0\label{lem:denom}
\end{align}
\end{lemma}

\begin{proof}
To establish \eqref{lem:betay} and \eqref{lem:discr} we note that \begin{align*}
(1+(m+n)y)^2-4mny^2=&1+2(m+n)y+(m-n)^2y^2>1+(m+n)y>1.
\end{align*}

The first inequality in \eqref{lem:denom} is clear, while the second is true because 
\begin{align*}(1+(m+n)y)^2>&(1+(m+n)y)^2-4mny^2>1.\end{align*} Hence $1+(m+n)y>\alpy$.
\end{proof}

\begin{proposition} Let $g_k(y):=\sum_{c=0}^{k}p_{k,c}y^c$ denote the coefficient of $x^k$ in $F(x,y)$. Then
\begin{align*}
g_k(y)&=\frac{1}{\sqrt{(1+(m+n)y)^2-4mny^2}}\left[\left(\frac{2mny^2}{(1+(m+n)y)
-\sqrt{(1+(m+n)y)^2-4mny^2}}\right)^{k+1}\right.\\
&\hspace{5mm}\left.-\left(\frac{2mny^2}{(1+(m+n)y)+\sqrt{(1+(m+n)y)^2-4mny^2}}\right)^{k+1}\right].
\end{align*}
\end{proposition}
\begin{proof}
From Proposition \ref{prop:Fxy} we know that 
\[F(x,y)=\frac{1}{1-x-(m+n)xy+mnx^2y^2}=\frac{1}{mny^2}\cdot\frac{1}{x^2-\frac{1+(m+n)y}{mny^2}+\frac{1}{mny^2}}.\]
In order to expand $F(x,y)$ into a power series we will use  partial fraction decomposition, but first we must factor $x^2-\frac{1+(m+n)y}{mny^2}+\frac{1}{mny^2}$ into two linear factors. Using the quadratic formula yields
\[x^2-\frac{1+(m+n)y}{mny^2}+\frac{1}{mny^2}=(x-\lambda_1)(x-\lambda_2)\]
where 
\begin{align}
\lambda_1=\lambda_1(y)&=\frac{(1+(m+n)y)-\sqrt{(1+(m+n)y)^2-4mny^2}}{2mny^2}\\
\lambda_2=\lambda_2(y)&=\frac{(1+(m+n)y)+\sqrt{(1+(m+n)y)^2-4mny^2}}{2mny^2}.
\end{align}

Since the discriminant is positive, by Equation \eqref{lem:discr}, we can use 
 partial fraction decomposition
\[F(x,y)=\frac{1}{mny^2}\cdot\frac{1}{x^2-\frac{1+(m+n)y}{mny^2}+\frac{1}{mny^2}}=\frac{1}{mny^2}\cdot\left(\frac{A_1}{x-\lambda_1}+\frac{A_2}{x-\lambda_2}\right).\]

Solving for $A_1, A_2$:
\[1=A_1(x-\lambda_2)+A_2(x-\lambda_1)\]
If $x=\lambda_1$, then $1=A_1(\lambda_1-\lambda_2)$. Hence   $A_1 =\frac{1}{\lambda_1-\lambda_2}$
and 
\begin{align*}
\lambda_1-\lambda_2=&\left( \frac{(1+(m+n)y)-\sqrt{(1+(m+n)y)^2-4mny^2}}{2mny^2}\right)\\
&-\left( \frac{(1+(m+n)y)+\sqrt{(1+(m+n)y)^2-4mny^2}}{2mny^2}\right)\\
=& -\frac{\sqrt{(1+(m+n)y)^2-4mny^2}}{mny^2}
\end{align*}
Thus $ A_1= \frac{-mny^2}{\sqrt{(1+(m+n)y)^2-4mny^2}}$.
Similarly, if $x=\lambda_2$, then $1=A_2(\lambda_1-\lambda_1)$. So   $A_2 =\frac{1}{\lambda_2-\lambda_1}=-A_1.$

Thus
\begin{align}
F(x,y)=&\frac{1}{mny^2}\cdot\left(\frac{-A_1}{\lambda_1-x}-\frac{A_2}{\lambda_1-x}\right)=
\frac{1}{mny^2}\cdot\left( \frac{-A_1}{\lambda_1}\sum_{i=0}^{\infty}\left(\frac{x}{\lambda_1}\right)^i-\frac{A_2}{\lambda_2}\sum_{i=0}^{\infty}\left(\frac{x}{\lambda_2}\right)^i\right).\label{eqinprevious}
\end{align}
If $g_k(y)$ denotes the coefficient of $x^k$ in $F(x,y)$, then using Equation \eqref{eqinprevious} we have that
\begin{align*}
g_k(y)=&\frac{1}{mny^2}\cdot\left(\frac{-A_1}{\lambda_1}\left(\frac{1}{\lambda_1}\right)^k-\frac{A_2}{\lambda_2}\left(\frac{1}{\lambda_2}\right)^k\right)\\
=& \frac{1}{\lambda_1 \sqrt{(1+(m+n)y)^2-4mny^2}} \left(\frac{2(mny^2)}{(1+(m+n)y)-\sqrt{(1+(m+n)y)^2-4(nmy^2)}}\right)^k\\
& +\frac{-1}{\lambda_2 \sqrt{(1+(m+n)y)^2-4mny^2}} \left(\frac{2(mny^2)}{(1+(m+n)y)+\sqrt{(1+(m+n)y)^2-4(nmy^2)}}\right)^k.\qedhere
\end{align*}
\end{proof}

To complete the proof of Theorem \ref{thm:gaussian} we make use the following result from \cite{DDKMV}.

\begin{theorem} \label{thm:DDKMV}\cite{DDKMV}*{Theorem 1.8} Let $\kappa$ be a fixed positive integer. For each $n$, let a discrete random variable $Y_n$ in $I_{n}=\{1,2,\ldots,n\}$  have
\begin{align*}
{\rm Prob}(Y_n=j)\ = \ \begin{cases}p_{j,n}/\sum_{j=1}^np_{j,n}&\text{{\rm if} $j\in I_n$}\\ 0&\text{{\rm otherwise}}\end{cases}
\end{align*}
for some positive real numbers $p_{1,n}, p_{2,n}, \ldots, p_{n,n}$. Let $g_n(y):=\sum_j p_{j,n}y^j$.

If $g_n$ has the form $g_n(y) = \sum_{i=1}^\kappa q_i(y)\alpha_i^n(y)$ where
\begin{enumerate}
\item[(i)]  for each $i \in \{1, \ldots, \kappa\}, q_i, \alpha_i: \mathbb{R} \to \mathbb{R}$ are three times differentiable functions which do not depend on $n$;
\item[(ii)]  there exists some small positive $\epsilon$ and some positive constant $\lambda < 1$ such that for all $y \in I_{\epsilon} = [1-\epsilon, 1 + \epsilon], |\alpha_1(y)| > 1$ and $|\frac{\alpha_i(y)}{\alpha_1(y)}| < \lambda < 1$ for all $i=2, \ldots, \kappa$;
\end{enumerate}
then
the mean $\mu_n$ and variance $\sigma_n^2$ of $Y_n$ both grow linearly with $n$. Specifically,
\begin{align*}
\mu_n \ =\ Cn + d + o(1), \ \ \ \ \sigma_n^2 \ =\ C^\prime n + d^\prime + o(1)
\end{align*}
where
\begin{align*}
C&\ = \  \frac{\alpha_1'(1)}{\alpha_1(1)}, \  d \ = \  \frac{q_1'(1)}{q_1(1)} \nonumber\\
C^\prime &\ = \  \frac{d}{dy}\left. \left(\frac{y\alpha_1'(y)}{\alpha_1(y)} \right)\right\vert_{y=1} \ = \  \frac{\alpha_1(1)[\alpha_1'(1)+ \alpha_1''(1)]-\alpha_1'(1)^2}{\alpha_1(1)^2}\nonumber\\
d^\prime &\ = \  \frac{d}{dy}  \left. \left(\frac{yq_1'(y)}{q_1(y)} \right) \right\vert_{y=1} \ = \  \frac{q_1(1)[q_1'(1)+ q_1''(1)]-q_1'(1)^2}{q_1(1)^2}.
\end{align*}
Moreover, if
\begin{enumerate}
\item[(iii)]    $\alpha_1'(1) \neq 0$ and $\frac{d}{dy}\left[ \frac{y\alpha_1'(y)}{\alpha_1(y)}\right]|_{y=1} \neq 0$, i.e., $C,C'>0$,
\end{enumerate}
then
as $n \to \infty$, $Y_n$ converges in distribution to a normal distribution.
\end{theorem}

Throughout the following proof we will simplify some calculations with the substitutions:
\[s=m+n,\quad p=mn, \quad \mbox{and}\quad \beta = \alp.\]

\begin{proof}[Proof of Theorem \ref{thm:gaussian}]
To prove Gaussian behavior we need only show that $g_k(y)$ satisfies the hypothesis of Theorem \ref{thm:DDKMV}.
Note that \[g_k(y)=q_1(y)\alpha_1^k(y) + q_2(y)\alpha_2^k(y),\]
where 
\[  q_i(y)=
\frac{(-1)^{i+1}2mny^2}{\left(1+(m+n)y+(-1)^i \sqrt{(1+(m+n)y)^2-4mny^2}\right)\sqrt{(1+(m+n)y)^2-4mny^2}}\]
and
\[  \alpha_i(y)=
\frac{2mny^2}{1+(m+n)y+(-1)^i \sqrt{(1+(m+n)y)^2-4mny^2}}.\]

\begin{itemize}
\item Condition (i): For each $i =1,2$, the functions  $q_i(y)$ and $\alpha_i(y)$ are three times differentiable.

\item Condition (ii): Let $\epsilon$ be some small positive constant and assume  $y \in I_{\epsilon} = [1-\epsilon, 1 + \epsilon]$. 

By  Equation \eqref{lem:denom}, we see that $0<\alpha_2(y)<\alpha_1(y)$. Thus for some positive constant $\lambda$, $|\frac{\alpha_2(y)}{\alpha_1(y)}| < \lambda < 1$.
Next we show that $\alpha_1(y)>1$. We begin by noting that $py^2 >0$ and  $\sqrt{(1+sy)^2-4py^2}>1$ (by equation \eqref{lem:discr}).
Hence
\begin{align*}
0<&  4py^2(py^2+\sqrt{(1+sy)^2-4py^2}-1)\\
(1+sy)^2<&  4py^2(py^2+\sqrt{(1+sy)^2-4py^2}-1)+(1+sy)^2\\
(1+sy)^2<& 4p^2y^4+4py^2\sqrt{(1+sy)^2-4py^2}+(1+sy)^2 - 4py^2\\
(1+sy)^2<& (2py^2+\sqrt{(1+sy)^2-4py^2})^2\\
1+sy<&  2py^2+\sqrt{(1+sy)^2-4py^2}\\
1<&\frac{2py^2}{1+sy-\sqrt{(1+sy)^2-4py^2}}.
\end{align*}

\item Condition (iii): First we compute $C=\frac{\alpha_1'(1)}{\alpha_1(1)}$ and prove that it is not 0.
Use 
\[  \alpha_1(y) = \frac{2py^2}{1+sy-\sqrt{(1+sy)^2-4py^2}}  \]
and compute
\begin{align*}
\alpha_1'(y) =&\frac{4py}{1+sy-\sqrt{(1+sy)^2-4py^2}}-\frac{2py^2\left[s-\tfrac{1}{2}\left( (1+sy)^2-4py^2 \right)^{-1/2}(2s(1+sy)-8py)     \right]}{(1+sy-\sqrt{(1+sy)^2-4py^2})^2}.
\end{align*}
Substitute $y=1$, use a common denominator to add fractions, and the numerator of $\alpha_1'(1)$  simplifies to 
\begin{align*}
4p(1+s-\beta)-2p\left[s-\frac{2s(1+s)-8p}{2\beta} \right]
=&2p\left( 2(1+s-\beta) -s+\frac{s(1+s)-4p}{\beta}   \right)\\
=&\frac{2p}{\beta}(1+s-\beta)(\beta-1).
\end{align*}
Hence \[C=\frac{\alpha_1'(1)}{\alpha_1(1)}=\frac{\frac{2p(1+s-\beta)(\beta-1)}{\beta(1+s-\beta)^2}}
{\frac{2p}{1+s-\beta}}=\frac{\beta-1}{\beta}=\frac{\alp-1}{\alp}.
\]
Note that this final value is positive (in particular not zero)  (see Equation \eqref{lem:discr}).

Second we compute $C'=\frac{\alpha_1'(1)-\alpha_1''(1)}{\alpha_1(1)}-\left(\frac{\alpha_1'(1)}{\alpha_1(1)}\right)^2$  and prove that it is not 0. Note
\begin{align*}
\alpha_1''(1)=&\frac{4p\left(s+\frac{4p-s(1+s)}{\beta}\right)^2}{(1+s-\beta)^3}-\frac{8p\left(s+\frac{4p-s(1+s)}{\beta}\right)}{(1+s-\beta)^2}+\frac{4p}{1-s-\beta}
-\frac{2p\left(\frac{(-4p+s(1+s))^2}{\beta^3}+\frac{4p-s^2}{\beta}\right)}{(1+s-\beta)^2}\\
=&\frac{4p}{1+s-\beta}\left(   \frac{4p-s-s^2-\beta-4p+1+2s+s^2}{\beta(1+s-\beta)}  \right)^2
-\frac{2p}{(1+s-\beta)^2}\frac{4p}{\beta^3}\\
=&\frac{4p}{(1+s-\beta)\beta^2}-\frac{8p^2}{(1+s-\beta)^2\beta^3}
\end{align*}
and using this we find that 
\begin{align*}
\frac{\alpha_1'(1)-\alpha_1''(1)}{\alpha_1(1)}=&\frac{\frac{2p(\beta-1)}{\beta(1+s-\beta)}+\frac{4p}{(1+s-\beta)\beta^2}-\frac{8p^2}{(1+s-\beta)^2\beta^3}}
{\frac{2p}{1+s-\beta}}
=\frac{\beta-1}{\beta}+\frac{\beta-1-s}{\beta^3}.
\end{align*}
Finally 
\begin{align}
C'=\frac{\alpha_1'(1)-\alpha_1''(1)}{\alpha_1(1)}-\left(\frac{\alpha_1'(1)}{\alpha_1(1)}\right)^2=&\frac{\beta-1}{\beta}+\frac{\beta-1-s}{\beta^3}-\left(\frac{\beta-1}{\beta}\right)^2\\
=&\frac{\beta^2-1-s}{\beta^3}\label{eqn:g0}\\
=&\frac{s(1+s)-4p}{\beta^3}.
\end{align}
By considering Equation \eqref{eqn:g0} with \eqref{lem:betay} we see that $C'>0$.
\end{itemize}
Therefore, by satisfying the conditions of Theorem \ref{thm:DDKMV}, we have completed our proof.
\end{proof}

\section{Directions for future research}
In this paper we considered the construction of $(n,m)$-bin sequences. For $d\in\mathbb{Z}_+$, one natural extension is to consider ${\bf{N}}=(n_1,n_2,\ldots,n_d)\in\mathbb{Z}_+^d$ and define ${\bf{N}}$-bin sequences in an analogous way to that of $(n,m)$-bin sequences. One could then study the ${\bf{N}}$-bin decompositions of positive integers. Namely, do these decompositions exist and are they unique? What is the behavior of the average number of summands used in the ${\bf{N}}$-bin legal decompositions, i.e.~is it Gaussian? 

Another further generalization would be to consider introducing a new parameter $s\in\mathbb{N}$ which accounts for the number of bins which must be skipped between summands used in a legal ${\bf{N}}$-bin decomposition. We call such decompositions the $(s,{\bf{N}})$-bin with skip decompositions. Note that when $s=1$ and ${\bf{N}}=(n,m)$, the $(s,{\bf{N}})$-bin with skip decompositions are exactly the $(n,m)$-bin decompositions and when $s\in\mathbb{Z}_+$ and ${\bf{N}}=b\in\mathbb{Z}_+$, the $(s,{\bf{N}})$-bin with skip decompositions are exactly the $(s,b)$-Generacci decompositions. Therefore the study of the $(s,{\bf{N}})$-bin with skip decompositions provides natural ways to generalize prior results in this area.

\appendix
\section{Negative Coefficient in Linear Reccurence}\label{appendix}

\begin{proposition}
The $(2,3)$-bin sequence is not a Positive Linear Recurrence Sequence (PLRS). 
\end{proposition}

\begin{proof}
By Equation \eqref{eq:single recurrence} the recurrence relation for the $(2,3)$-bin sequence is $a_x=4a_{x-3}-2a_{x-6}$. This has characteristic equation $y^6-4y^3+2$. By Eisenstein's criterion the polynomial $y^6-4y^3+2$ is irreducible in $\mathbb{Q}[y]$ since there exists a prime $p=2$ such that $p$ divides all non-leading coefficients of the polynomial, does not divide the leading coefficient, and whose square does not divide the constant term. Thus the polynomial $y^6-4y^3+2$ can not be factored into the product of non-constant polynomials with rational coefficients. Moreover, since this equation is irreducible in $\mathbb{Q}[y]$ our recurrence relation is minimal.   By applying Lemma B.1 in \cite{DDKMMV}, it is enough to show that all multiples of the characteristic equation cannot have the form
\[  y^{k+6} - \sum_{i=0}^{k+5} c_iy^i \]
with all $c_i>0$.

Consider the multiple of the characteristic equation (with $p_k\neq 0$):
\begin{align*}
    \sum_{i=0}^{k+6}c_iy^i &= \ \left(\sum_{j=0}^k p_j y^j\right)\left(y^6-4y^3+2\right)\\
    &=\ \sum_{i=0}^{k+6}\left( p_{i-6}-4p_{i-3}+2p_i  \right)y^i
\end{align*}
Thus $c_i =p_{i-6}-4p_{i-3}+2p_i $. Note that $p_i = 0$ when $i<0$ and when $i>k$. 

We will proceed by contradiction. Hence we assume $c_{k+6}>0$, and $c_i\leq 0$ whenever $i<k+6$.
Let $t$ be the smallest non-negative integer such that $p_t\neq 0$. Note that $0\leq t\leq k$.

\bigskip
We claim that for all integers $j\geq 0$ with $t+3j<k+6$, $p_{t+3j}<p_{t+3j-3}$ and $p_{t+3j}<0$. In other words the coefficients become increasingly negative. 
The proof of this claim is by induction. 

\noindent {\bf Base case $n=0$:}
By definition of $t$, $c_t =p_{t-6}-4p_{t-3}+2p_t=2p_t $. Hence $2p_t=c_t<0$, because $p_t\neq0$ and $t<k+6$. Thus $p_t<0=p_{t-3}$.

\noindent {\bf Base case $n=1$:} We have 
\begin{align*}
    c_{t+3} = p_{t-3} -4p_t +2p_{t+3} & \leq \ 0\\
    2p_{t+3} &\leq \ 4p_t\\
    p_{t+3} &\leq \ 2p_t <p_t
\end{align*}
where the last inequality is true because $p_t<0$.

\noindent {\bf Inductive Step:} We have
\begin{align}
    c_{t+3j} = p_{t+3j-6} -4p_{t+3j-3} +2p_{t+3j} & \leq \ 0\label{eq:contr_assume}\\
    2p_{t+3j} & \leq \ 4p_{t+3j-3} - p_{t+3j-6}\\
    2p_{t+3j} & \leq \ 4p_{t+3j-3} - p_{t+3j-3}\label{eq:ind_assum}\\
    p_{t+3j} & \leq \ 1.5p_{t+3j-3} \\
    p_{t+3j} & \leq \ p_{t+3j-3} \label{eq:zero}
\end{align}

Step \eqref{eq:contr_assume} is true because $t+3j<k+6$. Step \eqref{eq:ind_assum} is true by the inductive assumption. Finally step \eqref{eq:zero} is true because $p_{t+3j-3}<0$.

\bigskip

To establish our contradiction, choose $j^*$ such that $k<t+3j^*<k+6$.
Thus we have 
\begin{align}
    c_{t+3j^*} = p_{t+3j^*-6} -4p_{t+3j^*-3} +2p_{t+3j^*} & \leq \ 0\label{eq:contr_assume2}\\
    p_{t+3j^*-6} & \leq \ 4p_{t+3j^*-3} \label{eq:t}\\
     p_{t+3j^*-6} & \leq \ p_{t+3j^*-3} \label{eq:zero2}
\end{align}
Step \eqref{eq:contr_assume2} is true because $t+3j^*<k+6$. Step \eqref{eq:t} is true because $p_i=0$ when $i>k$. Step \eqref{eq:zero2} is true because $p_{t+3j^*-3}<0$. But this last line  contradictions  the claim we just proved above. 
 \end{proof}


\begin{thebibliography}{99}
\bibitem{CFHMN1} M. Catral, P. Ford, P. E. Harris, S. J. Miller, and D. Nelson,  \textit{Generalizing Zeckendorf's Theorem: The Kentucky Sequence}, Fibonacci Quarterly {\textbf{52}} (2014), no. 5, 68--90.

\bibitem{CFHMN2}M. Catral, P. Ford, P. E. Harris, S. J. Miller, and D. Nelson,  \textit{Legal Decompositions Arising from Non-positive Linear Recurrences}, The Fibonacci Quarterly {\textbf{54}} (2016), no. 4, 348--365.

\bibitem{newbehavior}M. Catral, P. Ford, P. E. Harris, S. J. Miller, D. Nelson, Z. Pan, and H. Xu, \textit{New Behavior in Legal Decompositions Arising from Non-positive Linear Recurrences}, The Fibonacci Quarterly {\textbf{55}} (2017), no. 3, 252--275.

\bibitem{mgonpaper}R. Dorward, P. Ford, E. Fourakis, P. E. Harris, S. J. Miller, E. Palsson, and H. Paugh, \textit{A Generalization of Zeckendorf's Theorem via Circumscribed $m$-gons}, Involve, a Journal of Mathematics, 2017, vol. 10, no. 1, pp. 125--150.


\bibitem{individualgaps}R. Dorward, P. Ford, E. Fourakis, P. E. Harris, S. J. Miller, E. Palsson, and H. Paugh, \textit{Individual Gap Measures from Generalized Zeckendorf Decompositions}, Uniform Distribution Theory {\textbf{12}} (2017), no. 1, 27--36.




\bibitem{DDKMMV}
P.~Demontigny, T.~Do, A.~Kulkarni, S.~J.~Miller, D.~Moon and U.~Varma, \textit{Generalizing Zeckendorf's Theorem to $f$-decompositions}, Journal of Number Theory, {\textbf{141}} (2014), 136--158.



\bibitem{DDKMV}
P.~Demontigny, T.~Do, A.~Kulkarni, S.~J.~Miller and U.~Varma, \textit{A Generalization of Fibonacci Far-Difference Representations and Gaussian Behavior}, Fibonacci Quarterly, {\textbf{52.3}} (2014), 247--273.


\bibitem{keller} T. J. Keller. \textit{Generalizations of Zeckendorf's Theorem}. Fibonacci Quarterly {\textbf{10}} (1975).

\bibitem{KKMW} M. Kologlu, G. Kopp, S. J. Miller and Y. Wang. \textit{On the number of summands in Zeckendorf decompositions}, Fibonacci Quarterly {\textbf{49}} (2011), no. 2, 116--130.

\bibitem{Lek}
C. G. Lekkerkerker, \textit{Voorstelling van natuurlyke getallen door een som van getallen van Fibonacci}, Simon Stevin {\bf{29}} (1951-1952), 190--195.

\bibitem{lengyel} T. Lengyel. \textit{A Counting Based Proof of the Generalized ZeckendorfÕs Theorem.} Mathematics Department, Occidental College, 2005. 

\bibitem{miller} S. J. Miller and Y. Wang. \textit{Gaussian Behavior in Generalized Zeckendorf Decomposition.} arXiv: 1107.2718 (2011).

\bibitem{Ze}
E. Zeckendorf, \textit{Repr\'esentation des nombres naturels par une somme des nombres de Fibonacci ou de nombres de Lucas}, Bulletin de la Soci\'et\'e Royale des Sciences de Li\'ege {\textbf{41}} (1972), pages 179--182.

\end{thebibliography}
\end{document}